\documentclass[11pt]{amsart}
\usepackage[margin=1.15in]{geometry}
\usepackage{amscd,amssymb, amsmath, wasysym}
\usepackage{graphicx}
\usepackage{amsfonts}
\usepackage{mathrsfs}    
\usepackage{amsmath}    
\usepackage{amsthm}     
\usepackage{amscd}      
\usepackage{amssymb}    
\usepackage{eucal}      
\usepackage{latexsym}   
\usepackage{graphicx}   
\usepackage{verbatim}   
\usepackage[all]{xy}     

\makeatletter

\newcounter{thmcounter}

\numberwithin{thmcounter}{section}
\numberwithin{equation}{thmcounter}

\newtheorem{theorem}[thmcounter]{Theorem}
\newtheorem{proposition}[thmcounter]{Proposition}
\newtheorem{lemma}[thmcounter]{Lemma}
\newtheorem{corollary}[thmcounter]{Corollary}

\theoremstyle{definition}
\newtheorem{definition}[thmcounter]{Definition}

\newtheorem{example}[thmcounter]{Example}

\newtheorem{remark}[thmcounter]{Remark}

\newtheorem{conjecture}[thmcounter]{Conjecture}

\newtheoremstyle{claim}{9pt}{3pt}{}{\parindent}{\bf}{.}{1em}{}

\theoremstyle{claim}
\newtheorem{claim}[equation]{Claim}

\newenvironment{namelist}[1]{%
\begin{list}{}
{
\settowidth{\labelwidth}{#1}%
\setlength{\labelsep}{0.3em}%
\setlength{\leftmargin}{\labelwidth}%
\addtolength{\leftmargin}{\labelsep}}}{%
\end{list}}

\newcommand{\nZ}{\mathbb{Z}}                     
\newcommand{\nR}{\mathbb{R}}                     
\newcommand{\nC}{\mathbb{C}}                     
\newcommand{\nQ}{\mathbb{Q}}                     

\newcommand{\nP}{\mathbb{P}}                     

\newcommand{\nA}{\mathbb{A}}                     

\newcommand{\sO}{\mathscr{O}}                    

\newcommand{\sI}{\mathscr{I}}                    

\newcommand{\mf}[1]{\mathfrak{#1}}

\DeclareMathOperator{\ann}{ann}                  

\DeclareMathOperator{\Bl}{Bl}                    

\DeclareMathOperator{\Char}{char}                
\DeclareMathOperator{\codim}{codim}              

\DeclareMathOperator{\Fitt}{Fitt}                

\DeclareMathOperator{\Hom}{Hom}                  

\DeclareMathOperator{\Jac}{Jac}                  

\DeclareMathOperator{\sHom}{\mathscr{H}om}       
\DeclareMathOperator{\Spec}{Spec}                
\DeclareMathOperator{\spec}{Spec}                

\DeclareMathOperator{\pd}{pd}                    
\DeclareMathOperator{\Ext}{Ext}                  

\newcounter{rkcounter}             
\begin{document}

\title[Mather-Jacobian multiplier ideals]{A note on Mather-Jacobian multiplier ideals}
\author{Wenbo Niu and Bernd Ulrich}

\address{Department of Mathematics, Purdue University, West Lafayette, IN 47907-2067, USA}
\email{niu6@math.purdue.edu}

\address{Department of Mathematics, Purdue University, West Lafayette, IN 47907-2067, USA}
\email{ulrich@math.purdue.edu}



\begin{abstract} By using Mather-Jacobian multiplier ideals, we first prove a formula on comparing Grauert-Riemenschneider canonical sheaf with canonical sheaf of a variety over an algebraically closed field of characteristic zero. Then we turn to study Mather-Jacobian multiplier ideals on algebraic curve, in which case the definition of Mather-Jacobian multiplier ideal can be extended to a ground field of any characteristic. We show that Mather-Jacobian multiplier ideal on curves is essentially the same as an integrally closed ideal. Finally by comparing conductor ideal with Mather-Jacobian multiplier ideal, we give a criterion when an algebraic curve is a locally complete intersection.
\end{abstract}

\maketitle

\section{Introduction}
\noindent Recently, the theory of Mather-Jacobian multiplier ideals (MJ-multiplier for short) on arbitrary varieties over an algebraically closed field of characteristic zero has been developed by Ein-Ishii-Mustata \cite{Ein:MultIdeaMatherDis} and Ein-Ishii \cite{Ein:SingMJDiscrapency} (see also de Fernex-Docampo \cite{Roi:JDiscrepancy} for a similar theory on normal varieties). This new notion generalizes the classical theory of multiplier ideals on nonsingular varieties (or normal $\nQ$-Gorenstein varieties) and has found some interesting applications.  Throughout this paper, the ground field $k$ is always assumed to be {\em algebraically closed}. By a variety we mean a reduced equidimensional separated scheme of finite type over $k$.

We first prove the following result on comparing Grauert-Riemenschneider canonical sheaf with canonical sheaf of a variety by using MJ-multiplier ideals. This result partially generalizes a formula established by de Fernex-Docampo in \cite[Theorem C]{Roi:JDiscrepancy} on a normal variety.

\begin{theorem} Let $X$ be a variety over $k$ of characteristic zero. Let $\omega_X$ be the canonical sheaf, $\omega^{GR}_X$ the Grauert-
Riemenschneider canonical sheaf, and $\widehat{\sI}(\sO_X)$ the Mather-Jacobian multiplier ideal. Then
$$\widehat{\sI}(\sO_X)\cdot\omega_X\subseteq \omega^{GR}_X.$$
\end{theorem}

\noindent It should be mentioned that the formula of de Fernex-Docampo on a normal variety is stronger than the above one.  However, their formula seems hard to be established on any variety. We also mention that when $X$ is a locally complete intersection, it is well-known to experts that $\widehat{\sI}(\sO_X)\cdot\omega_X=\omega^{GR}_X$. (See Remark \ref{rmk:01} for details.)

Next we turn to understand MJ-multiplier ideals on algebraic curves, which is the first case one should investigate. Let $X$ be an algebraic curve, i.e., $X$ is a dimension one variety over $k$. We first show that the definition of MJ-multiplier ideal can be extended for $k$ of any characteristic, see Proposition \ref{p:10}. Simply by the definition, any MJ-multiplier ideal is automatically integrally closed and inside $\widehat{\sI}(\sO_X)$. We then show that this fact is actually describing MJ-multiplier ideals on curves completely.

\begin{theorem}\label{p:01} Let $X$ be an algebraic curve over $k$ of any characteristic and $\mf{a}\subseteq\sO_X$ be an ideal. Then $\mf{a}$ is a Mather-Jacobian multiplier ideal if and only if $\mf{a}$ is integrally closed and $\mf{a}\subseteq\widehat{\sI}(\sO_X)$.
\end{theorem}

Note that our theorem is global in nature. The question that whether an integrally closed ideal is a multiplier ideal was initially raised by Lipman-Watanabe \cite{Lipman:IntClosedIdealMultipler}. They proved that on a nonsingular surface, locally a multiplier ideal is the same as an integrally closed ideal. It was also studied in the work of Favre-Jonsson \cite{Favre:ValMultiplierIdeals}. Later, it was generalized to log terminal surface by Tucker \cite{Tucker:IntIdealMulitplier}. Finally the question was completely solved by the celebrated work of Lazarsfeld-Lee \cite{Lazarsfeld:LocalSyzyMultiplier} for higher dimensional nonsingular varieties. However, as complementary, we still wonder the same question on any surface for MJ-multiplier ideals, see Conjecture \ref{p:04}.

On the curve $X$, there are three intrinsic ideals: the Jacobian ideal $\Jac_X$, MJ-multiplier ideal $\widehat{\sI}(\sO_X)$, and the conductor ideal $\mf{C}_X$. They all capture the singularities of $X$ and have the inclusion $\overline{\Jac}_X\subseteq \widehat{\sI}(\sO_X)\subseteq\mf{C}_X$. (This inclusion is also true for higher dimensional varieties.) Our last theorem shows a criterion for a curve to be a local complete intersection by using these ideals.

\begin{theorem}\label{p:02} Let $X$ be an algebraic curve over $k$ of any characteristic. Let $\widehat{\sI}(\sO_X)$ be the Mather-Jacobian multiplier ideal and $\mf{C}_X$ be the conductor ideal. Then $X$ is a local complete intersection if and only if $\widehat{\sI}(\sO_X)=\mf{C}_X$.
\end{theorem}

\noindent We point out that the above theorem is not true in general for higher dimensional varieties, see Example \ref{eg:01}. But it can be established for codimension one points of any variety, see Remark \ref{rmk:03}. It would be interesting to understand algebraic or geometric consequences when conductor ideal is the same as MJ-multiplier ideal for any variety. We prove some evidences in Proposition \ref{p:06} along this direction.

\vspace{0.3cm}
{\em Acknowledgement}. We are grateful to Lawrence Ein and Shihoko Ishii for valuable discussions.

\section{Mather-Jacobian Multiplier ideals}

\noindent Throughout this section, we assume the ground field $k$ is of {\em characteristic zero}. We start by recalling the definition of MJ-multiplier ideal defined in \cite{Ein:MultIdeaMatherDis}. For more detailed discussion, we refer to the paper \cite{Ein:MultIdeaMatherDis}.

\begin{definition} Let $X$ be a variety over $k$ of dimension $d$ and $f:Y\longrightarrow X$ be a resolution of singularities factoring through the Nash blow-up of $X$. Then the image of the canonical homomorphism
$$f^*(\wedge^d\Omega^1_X)\longrightarrow \wedge^d\Omega^1_X$$
is an invertible sheaf of the form $\Jac_f\cdot\wedge^d\Omega^1_X$ where $\Jac_f$ is the relative Jacobian which is invertible and defines an effective divisor which is called the {\em Mather discrepancy divisor} and denoted by $\widehat{K}_{Y/X}$.
\end{definition}

\begin{definition} Let $X$ be a variety over $k$ and $\mf{a}\subseteq \sO_X$ a nonzero ideal on $X$. Given a log resolution $f:Y\longrightarrow X$ of $X$ and $\Jac_X\cdot\mf{a}$ such that $\mf{a}\cdot\sO_Y=\sO_Y(-Z)$ and $\Jac_X\cdot\sO_Y=\sO_Y(-J_{Y/X})$ for some effective divisors $Z$ and $J_{Y/X}$ on $Y$ (such resolution automatically factors through the Nash blow-up, see Remark 2.3 of \cite{Ein:MultIdeaMatherDis}). The {\em Mather-Jacobian multiplier ideal} of $\mf{a}$ of exponent $t\in \nR_{\geq 0}$ is defined by
$$\widehat{\sI}(X,\mf{a}^t):=f_*\sO_Y(\widehat{K}_{Y/X}-J_{Y/X}-\lfloor tZ\rfloor),$$
where $\lfloor\ \ \rfloor$ means the round down of an $\nR$-divisor. Sometimes we simply write it as $\widehat{\sI}(\mf{a}^t)$ and call it as {\em MJ-multiplier ideal}.
\end{definition}

\begin{remark}\label{rmk:02} (1) It has been showed in \cite[Corrolary 2.14]{Ein:MultIdeaMatherDis} that the sheaf $\widehat{\sI}(X,\mf{a}^t)$ defined above is actually an ideal of $\sO_X$.

(2) If $X$ is a curve, then we choose its normalization $f: X'\longrightarrow X$ as a resolution of singularities. In this case, we do not need to use Nash blow-up.

(3) It is clear that any MJ-multiplier ideal is contained in the ideal $\widehat{\sI}(\sO_X)$. It would be interesting to understand more about this intrinsic ideal $\widehat{\sI}(\sO_X)$.
\end{remark}

For a variety $X$, let $f:X'\longrightarrow X$ be a resolution of singularities (or simply take $X'$ to be the normalization of $X$). Then the conductor ideal $\mf{C}_X$ is defined to be
$$\mf{C}_X:=\sHom_{\sO_X}(f_*\sO_{X'},\sO_X)=\ann(f_*\sO_{X'}/\sO_X).$$

The following easy proposition shows that MJ-multiplier ideals are in fact contained in the conductor ideal. See also \cite[Corrolary 2.14]{Ein:MultIdeaMatherDis}.
\begin{proposition}\label{p:10} Let $X$ be a variety over $k$ and $\mu:X'\rightarrow X$ be the normalization. Then one has $\widehat{\sI}(\sO_X)\cdot\mu_*\sO_{X'}\subseteq \sO_X$. In particular, $\widehat{\sI}(\sO_X)\subseteq \mf{C}_X$.
\end{proposition}
\begin{proof} Since $\widehat{\sI}(\sO_X)$ is an ideal both in $\sO_X$ and $\mu_*\sO_{X'}$, the result then follows from the fact that $\mf{C}_X$ is the maximal ideal both in $\sO_X$ and $\mu_*\sO_{X'}$.
\end{proof}

\begin{theorem}\label{p:07} Let $X$ be a variety over $k$. Let $\omega_X$ be its canonical sheaf, $\omega^{GR}_X$ its Grauert-
Riemenschneider canonical sheaf, and $\widehat{\sI}(\sO_X)$ its Mather-Jacobian multiplier ideal. Then one has
$$\widehat{\sI}(\sO_X)\cdot\omega_X\subseteq \omega^{GR}_X.$$
\end{theorem}
\begin{proof} The question is local, so we can assume that $X$ is affine and embedded in an affine space $\nA^N$ such that $d:=\dim X$ and $c:=\codim_{\nA^N}X$. Assume that $I_X:=(F_1,F_2,\cdots,F_r)$. We can take those generators $F_i$'s general such that any $c$ of them provide a general link of $X$. Specifically, let $J\subset \{1,2,\cdots, r\}$ such that $|J|=c$. Then let $I_{V_J}$ to be the ideal generated by $F_i\in J$, i.e., $I_{V_J}:=\{F_i|i\in J\}$, and let $V_J$ be the subscheme defined by $I_{V_J}$. Since we take $F_i$'s general, $V_J$ is a complete intersection in $\nA^N$. Denote $q_J:=(I_{V_J}:I_X)\cdot\sO_X$ and $\Jac_J:=\Jac_{V_J}\cdot \sO_X$. Consider the following morphisms
$$\wedge^d\Omega^1_X\stackrel{c_X}{\longrightarrow}\omega_X\stackrel{u_J}{\longrightarrow}\omega_{V_J}|_X\stackrel{w_J}{\longrightarrow}\sO_X.$$
It has been proved in \cite{Ein:JetSch} that
\begin{itemize}
\item [(1)] $u_J$ is an injective and $w_J$ is a canonical isomorphism;
\item [(2)] the image of $u_J$ is $q_J\otimes \omega_{V_J}|_X$ and therefore if set $\alpha_J:=w_J\circ u_J$, we get an isomorphism $$\alpha_J:\omega_X\longrightarrow q_J;$$
\item [(3)] the image of $u_J\circ c_X$ is $\Jac_J\otimes \omega_{V_J}|_X$ and under the isomorphism $\alpha_J$ above the image of $c_X$ is $\Jac_J$.
\end{itemize}
Now consider $J_0:=\{1,2,\cdots, c\}$ and write $V_0:=V_{J_0}$, $\Jac_0:=\Jac_{J_0}$, and $q_0=q_{J_0}$. From the surjective morphism
$$\wedge^d\Omega^1_X\longrightarrow \Jac_0\otimes \omega_{V_0}|_X,$$
we deduce that the Nash blowup $\widehat{X}=\Bl_{\Jac_0}X$ and $\widehat{K}_X=\Jac_0\cdot\sO_{\widehat{X}}\otimes \sigma^*\omega_{V_J}|_X$, where $\sigma:\widehat{X}\rightarrow X$ is the projection. Consider a log resolution $f:X'\rightarrow X$ of $\Jac_X$, $\Jac_0$ and $q_0$ such that $\Jac_X\cdot\sO_{X'}=\sO_{X'}(-J_{X'/X})$,  $\Jac_0\cdot\sO_{X'}=\sO_{X'}(-F_0)$, and $q_0\cdot\sO_{X'}=\sO_{X'}(-Q_0)$. This $f$ factors through $\widehat{X}$ and therefore we have
$$\widehat{K}_{X'/X}=\omega_{X'}\otimes \sO_{X'}(F)\otimes f^*\omega_{V_0}^{-1}|_X,$$
which implies that
\begin{equation}
\omega^{GR}_X=f_*\sO_{X'}(\widehat{K}_{X'/X}-F_0)\otimes \omega_{V_0}|_X.
\end{equation}
Note that $f_*\sO_{X'}(\widehat{K}_{X'/X}-F_0)$ is an ideal sheaf because it is inside $\widehat{\sI}(\sO_X)$ since $\Jac_0\subseteq \Jac_X$.
On the other hand, we have the equality $\omega_X=q_0\otimes \omega_{V_0}|_X$. Hence we deduce that
$$(\omega^{GR}_X:\omega_X)=(f_*\sO_{X'}(\widehat{K}_{X'/X}-F_0)\otimes \omega_{V_0}|_X:q_0\otimes \omega_{V_0}|_X)=(f_*\sO_{X'}(\widehat{K}_{X'/X}-F_0):q_0),$$
since $\omega_{V_0}|_X$ is invertible.

In order to prove the theorem we need to show $\widehat{\sI}(\sO_X)\cdot q_0\subseteq f_*\sO_{X'}(\widehat{K}_{X'/X}-F_0)$. But note that $\widehat{\sI}(\sO_X)\cdot q_0\subseteq \widehat{\sI}(q_0)=f_*\sO_{X'}(\widehat{K}_{X'/X}-J_{X'/X}-Q_0)$. Thus it suffices to show
\begin{equation}\label{eq:06}
\Jac_X\cdot q_0\subseteq \Jac_0.
\end{equation}
To this end, consider for any $J\subset \{1,2,\cdots,r\}$ with $|J|=c$ the isomorphism
$$\alpha_J\circ\alpha^{-1}_0: q_0\longrightarrow q_J.$$
Since ideals $q_0$ and $q_J$ contain some nonzero divisors of $\sO_X$, the isomorphism $\alpha_J\circ\alpha^{-1}_0$ is given by $\alpha_J\circ\alpha^{-1}_0(r)=\frac{b_J}{a_J}\cdot r$ for any $r\in q_0$, where $b_J$ and $a_J$ are some nonzero divisors in $\sO_X$. Thus we have $q_J=\frac{b_J}{a_J}\cdot q_0$ and $\Jac_J=\frac{b_J}{a_J}\cdot \Jac_0$. But notice that
$$\Jac_X=\sum_J \Jac_J,$$
and therefore $$\Jac_X\cdot q_0=\sum_J \Jac_J\cdot q_0=\sum_J \frac{b_J}{a_J}\cdot \Jac_0\cdot q_0=\sum_J \frac{b_J}{a_J}\cdot q_0\cdot \Jac_0=\sum_J q_J\cdot \Jac_0\subseteq \Jac_0.$$
This proves the inclusion (\ref{eq:06}) and the theorem then follows.
\end{proof}

\begin{remark}\label{rmk:01} (1) When $X$ is a locally complete intersection, the image of the canonical morphism $\wedge^d\Omega_X\longrightarrow \omega_X$
is $\Jac_X\otimes \omega_X$. Then following the same argument as above, one can deduce that
$$\widehat{\sI}(\sO_X)\cdot\omega_X=\omega^{GR}_X.$$
This equality was also mentioned explicitly or implicitly in work of \cite{Ein:MultIdeaMatherDis} and \cite{Roi:JDiscrepancy}.

(2) In their work \cite{Roi:JDiscrepancy}, when $X$ is normal, de Fernex-Docampo established that
$$\sI^{\diamond}(\mathfrak{d}^{-1}_X)\cdot \omega_X\subseteq\omega^{GR}_X,$$
where the ideal $\sI^{\diamond}(\mathfrak{d}^{-1}_X)$ is a special multiplier ideal defined on normal varieties. It is not clear to us right now that this ideal can be extended to any variety. It would be interesting if a similar result on any variety can be established. But so far it seems too strong to expect.
\end{remark}

\begin{example} It is easy to see that in general we do not have equality $\widehat{\sI}(\sO_X)\cdot\omega_X=\omega^{GR}_X$ for any variety. Indeed, we can take $X$ which has rational singularities but not MJ-canonical. Then $\omega_X=\omega^{GR}_X$ but $\widehat{\sI}(\sO_X)$ is not trivial.
\end{example}

\section{Mather-Jacobian multiplier ideals on curves}

\noindent In this section, we study MJ-multiplier ideals on an algebraic curve. We assume that the ground field $k$ is of {\em any characteristic}. We shall first prove that the definition of MJ-multiplier ideals still works for such ground field $k$.

We start with briefly recalling some facts about a general Noether normalization. Assume that $X=\Spec R$ is an affine variety of dimension $d$ and $R=k[x_1,\cdots,x_n]$ is a finitely generated $k$-algebra generated by $x_1,\cdots,x_n$. We can choose $x_i$'s general (for example, we take the form $a_1x_1+\cdots+a_nx_n$, where $a_i$'s are general elements in the ground field $k$) such that any $d$ of them, say $x_{j_1},\cdots, x_{j_d}$, will give a Noether normalization $k[x_{j_1},\cdots, x_{j_d}]\subseteq k[x_1,\cdots,x_n]$ and the total ring of quotients $Q(R)$ is \'etale over $k(x_{j_1},\cdots, x_{j_d})$. The following proposition is well-known to the experts and can be easily proved by doing calculation.

\begin{proposition}\label{p:03} Let $X=\Spec R$ be an affine variety of dimension $d$ over $k$, where $R=k[x_1,\cdots,x_n]$, and let $X'$ be the normalization of $X$. For any $J\subset \{1,\cdots,n\}$ with $|J|=d$, write $\textbf{x}_J=\{x_j\ | \ j\in J\}$ and $A_J:=\Spec k[\textbf{x}_J]$ and consider the following diagram

$$\xymatrix{
X' \ar[rr]\ar[dr] &   & X \ar[dl]\\
                         & A_J  &
}$$
Then we can take the generators $x_i$'s to be general such that we have
\begin{equation}\label{eq:13}
\Jac_X=\sum_{J}\Jac(X/A_J),\quad \mbox{and}\quad \Jac(X'/X)=\sum_{J}\Jac(X'/A_J),
\end{equation}
where $\Jac(X'/X)$, $\Jac(X/A_J)$ and $\Jac(X'/A_J)$ are relative Jacobian ideals in the diagram.
\end{proposition}

Now let $X$ be an algebraic curve over $k$ and let $f:X'\longrightarrow X$ be the normalization of $X$. Let $\Jac(X/k)$ be the Jacobian ideal of $X$ and we can write $\Jac(X/k)\cdot\sO_{X'}=\sO_{X'}(-J_{X'/X})$ for an effective divisor $J_{X'/X}$ on $X'$. Let $\Jac(X'/X)$ be the relative Jacobian ideal of $f$ and clearly we have $\Jac(X'/X)=\sO_{X'}(-\widehat{K}_{X'/X})$. For an ideal  $\mf{a}$ of $\sO_X$ and we write $\mf{a}\cdot\sO_{X'}=\sO_{X'}(-Z)$ where $Z$ is an effective divisor on $X'$. Then formally we can form an Mather-Jacobian multiplier ideal of $\mf{a}$ of exponent $t\in \nR_{\geq 0}$ by
$$\widehat{\sI}(X,\mf{a}^t):=f_*\sO_Y(\widehat{K}_{Y/X}-J_{Y/X}-\lfloor tZ\rfloor).$$
The crucial point is that we need to show this fraction ideal sheaf $\widehat{\sI}(X,\mf{a}^t)$ is indeed inside $\sO_X$ if $k$ has any characteristic. When $k$ has characteristic zero, it is proved in \cite{Ein:MultIdeaMatherDis}. Here we show it in the following proposition for $k$ of any characteristic. Note that it is enough to show that $\widehat{\sI}(\sO_X)$ is an ideal of $\sO_X$.

\begin{proposition}\label{p:11} As setting above, $\widehat{\sI}(\sO_X)$ is inside $\sO_X$.
\end{proposition}
\begin{proof} The question is local, so we assume that $X=\Spec R$ be an affine where $R=k[x_1,\cdots,x_n]$ and $x_i$'s are generators of $k$-algebra. We assume that the normalization $X'=\Spec S$. By Proposition \ref{p:03}, we can choose $x_i$'s general so that we have equalities in (\ref{eq:13}). All we need is to show $(S:\Jac(X'/X))\cdot \Jac(X/k)$ is an $R$-ideal. To this end, for a Noether normalization $A_J$ as in Proposition \ref{p:03}, by \cite[Theorem 2]{Lipman:JacBrianSkoda}, we see that
$$(S:\Jac(X'/A_J))\cdot\Jac(X/A_J)\subseteq \mf{C}_X,$$
where $\mf{C}_X$ is the conductor ideal of $X$. Now by (\ref{eq:13}) and notice that $(S:\Jac(X'/X))\subseteq(S:\Jac(X'/A_J))$ we see that
$$(S:\Jac(X'/X))\cdot\Jac(X/k)\subseteq \mf{C}_X.$$
Finally since the conductor ideal $\mf{C}_X$ is inside $R$ so the result follows.
\end{proof}

\begin{remark} When $\Char k=0$, \cite{Ein:MultIdeaMatherDis} has showed that $\widehat{\sI}(\sO_X)\subseteq \sO_X$ by reducing to the usual multiplier ideals on a nonsingular variety. From it, we proved Proposition \ref{p:10}. This method requires the existence of  resolution of singularities.

If $\Char k>0$, although we have proved the above proposition for curves, it is not clear to us how to extend the theory of MJ-multiplier ideal on varieties of any dimension. Hopefully, the above proposition could provide some evidence along this direction.
\end{remark}

\begin{theorem}\label{p:01} Let $X$ be an algebraic curve over $k$ and $\mf{a}\subseteq\sO_X$ be an ideal. Then $\mf{a}$ is a $MJ$-multiplier ideal if and only if $\mf{a}$ is integrally closed and $\mf{a}\subseteq\widehat{\sI}(\sO_X)$.
\end{theorem}
\begin{proof} The necessary part is clear by the definition of Mather-Jacobian multiplier ideals. So we prove the sufficient part by assuming that $\mf{a}$ is integrally closed and $\mf{a}\subseteq\widehat{\sI}(\sO_X)$.

Let $f:X'\longrightarrow X$ be the normalization of $X$. We write $\mf{a}\cdot\sO_{X'}=\sO_{X'}(-E)$ and decompose $E$ as $$E=\sum_{i=1}^ta_iE_i,$$ where $E_i$ are distinct prime divisors and $a_i>0 $.
Let $\widehat{K}_{X'/X}$ be the Mather discrepancy divisor. Let $\Jac_X$ be the Jacobian ideal of $X$, and write $\Jac_X\cdot\sO_{X'}=\sO_{X'}(-J_{X'/X})$. By assumption of $\mf{a}\subseteq \widehat{\sI}(\sO_X)$, we have an inequality
\begin{equation}\label{eq:12}
-E\leq \widehat{K}_{X'/X}-J_{X'/X},
\end{equation}
Furthermore, since $\mf{a}$ is integrally closed and is contained in $\widehat{\sI}(\sO_X)$, we have $\mf{a}=f_*\sO_{X'}(-E)$.

Thus we can rewrite
$$\widehat{K}_{X'/X}-J_{X'/X}=\sum_{i=1}^tr_iE_i+N, \ \ r_i\in \nZ,$$
where no $E_i$ is in the support of $N$ and $N$ must be effective because of (\ref{eq:12}). Now set
$$B:=\sum_{i=1}^t(r_i+a_i)E_i+N$$
Note that for any $E_i$, we have
$r_i+a_i\geq 0$ because of (\ref{eq:12}) and therefore $B$ is an effective divisor.
Consider a divisor
$$T:=\widehat{K}_{X'/X}-J_{X'/X}-B.$$
Thus $T=-E$ and therefore we see that $f_*\sO_{X'}(T)=f_*\sO_{X'}(-E)=\mf{a}$.

Next we shall find an ideal $\mf{b}$ and a number $\lambda$ such that $\mf{b}\cdot\sO_{X'}=\sO_{X'}(-B')$ and $B=\lfloor \lambda B'\rfloor$. Once we have this $\mf{b}$ and we immediately have $\mf{a}=\widehat{\sI}(b^{\lambda})$. Our strategy is that we construct $B'$ first and then push it down to get $\mf{b}$. For this, set
$$B'=nB+\sum_{j:r_j+a_j=0}a_jE_j.$$
We fix a number $n\gg0$ so that $n>a_i$ for all $i=1,...,t$. Now for such divisor $B'$, we have
$$\sO_{X'}(-B')\subseteq\sO_{X'}(-\sum_{i=1}^ta_iE_i).$$
This implies that
$$f_*\sO_{X'}(-B')\subseteq f_*\sO_{X'}(-E)=f_*\sO_{X'}(T)\subseteq \sO_X.$$
Thus we set $\mf{b}:=f_*\sO_{X'}(-B')$, which is an ideal. Now since $f$ is an affine finite morphism, we then get the surjection
\begin{equation}\label{eq:02}
f^*f_*\sO_{X'}(-B')\longrightarrow \sO_{X'}(-B')\longrightarrow 0.
\end{equation}
Thus we see that $\mf{b}\cdot\sO_{X'}=\sO_{X'}(-B')$. Finally, we set $\lambda=\frac{1}{n}$, and notice that
$B=\lfloor\lambda B'\rfloor$ and therefore we obtain $\mf{a}=\widehat{\sI}(\mf{b}^\lambda)$.
\end{proof}

\begin{remark} (1) Note that the theorem is global, i.e., $X$ need not to be affine. The similar result is certainly not true for higher dimensional variety in general by the result of Lazarsfeld-Lee \cite{Lazarsfeld:LocalSyzyMultiplier}. The essential point in the curve case is the surjectivity (\ref{eq:02}).

(2) However, for surface case, in the light of the work of Lipman-Watanabe \cite{Lipman:IntClosedIdealMultipler}, Favre-Jonsson \cite{Favre:ValMultiplierIdeals}, and Tucker \cite{Tucker:IntIdealMulitplier}, we may expect the following conjecture.
\end{remark}

\begin{conjecture}\label{p:04} Let $X$ be an algebraic surface over $k$ and $\mf{a}\subseteq\sO_X$ be an ideal. Let $\mf{p}\in X$ be a closed point. Then $\mf{a}_{\mf{p}}$ is a MJ-multiplier ideal at $\mf{p}$ if and only if $\mf{a}_{\mf{p}}$ is integrally closed and $\mf{a}_{\mf{p}}\subseteq\widehat{\sI}(\sO_X)_{\mf{p}}$.
\end{conjecture}

Next we prove Theorem \ref{p:02}. Let us first mention the following easy lemma used in the proof.

\begin{lemma}\label{p:09} Let $(R,\mf{m})$ be a Cohen-Macaulay local ring of dimension one with a canonical module $\omega_R$. Let $M$ be a finite torsion-free $R$-module and $N$ be a submodule of $M$ such that the length $\lambda(M/N)$ if finite. Then one has
$$\lambda(M/N)=\lambda(\Hom_R(N,\omega_R)/\Hom_R(M,\omega_R)).$$
\end{lemma}
\begin{proof} Apply $\Hom_R(\ \_\ ,\omega_R)$ to the exact sequence $0\rightarrow N\rightarrow M \rightarrow M/N\rightarrow 0$. By the assumption that $M/N$ has finite length and $M$ is torsion-free, we deduce an exact sequence
$$0\rightarrow \Hom_R(M,\omega_R)\rightarrow \Hom_R(N,\omega_R)\rightarrow \Ext^1_R(M/N,\omega_R)\rightarrow 0.$$
Now by local duality and Matlis duality, we can show that
$$\lambda(M/N)=\lambda(\Ext^1_R(M/N,\omega_R))$$
and then the result follows.
\end{proof}
\begin{theorem}\label{p:05} Let $X$ be an algebraic curve over $k$. Let $\widehat{\sI}(\sO_X)$ be the MJ-multiplier ideal and $\mf{C}_X$ be the conductor ideal. Then $X$ is a local complete intersection if and only if $\widehat{\sI}(\sO_X)=\mf{C}_X$.
\end{theorem}
\begin{proof} If $X$ is a local complete intersection, then the result follows from Remark \ref{rmk:01} (1) (which is still true in our case) and Lemma \ref{p:08}. So we just need to prove the sufficient part of the theorem by assuming that $\widehat{\sI}(\sO_X)=\mf{C}_X$.

The question is local, so we can assume that $X$ is an affine curve. Let $f:X' \longrightarrow X$ be the normalization of $X$. Note that since $f$ is finite, by Remark \ref{rmk:02} (1), the condition $\widehat{\sI}(\sO_X)=\mf{C}_X$ is the same as
\begin{equation}\label{eq:11}
\Jac_X\cdot \sO_{X'}=\mf{C}_X\cdot \Jac(X'/X),
\end{equation}
where $\Jac_X$ is the Jacobian ideal of $X$ and $\Jac(X'/X)$ is the relative Jacobian of $f$. Let $\mf{p}\in X$ be a closed point and $R:=\sO_{\mf{p}}$ be the local ring at $\mf{p}$. Consider the following fiber product
$$\begin{CD}
X'_{\mf{p}} @>>> X'\\
@VVV @VVf V\\
\Spec R @>>> X
\end{CD}$$
where $X'_{\mf{p}}:=\Spec R\times_X X'$ and we can write $X'_{\mf{p}}=\Spec S$. Note that $S$ is a regular semilocal noetherian ring and is a finitely generated $R$-module. Furthermore since $S$ is locally principal, it is then principal, i.e., any ideal of $S$ is generated by an element of $S$.

Now we shall take a Noether normalization $\mu: X\longrightarrow \nA=\spec k[y]$ for some $y$ in $\sO_X$. Write $\Jac(X/\nA)$ and $\Jac(X'/\nA)$ to be relative Jacobian ideals. We make the following crucial claim in our proof.
\begin{claim}\label{eq:07} we can choose $y$ general so that along $X'_{\mf{p}}$, we have
\begin{itemize}
\item [(1)] $\Jac(X'/X)\cdot S=\Jac(X'/\nA)\cdot S$.
\item [(2)] $(\Jac_X)_{\mf{p}}\cdot S=\Jac(X/\nA)_{\mf{p}}\cdot S$;
\end{itemize}
\end{claim}
\noindent{\em Proof of claim.} We can assume $X=\Spec k[x_1,\cdots, x_n]$, where $x_i$'s are generators of $k$-algebra. Consider the natural morphism
$$S\otimes_R\Omega^1_{R/k}\stackrel{u}{\longrightarrow} \Omega^1_{S/k}\rightarrow \Omega^1_{S/R}\rightarrow 0.$$
The image of $u$ is generated by $1\otimes dx_1,\cdots, 1\otimes dx_n$ as $S$-module since $\Omega^1_{R/k}$ is generated by those $dx_i$'s. Because $S$ is regular we see that $\Omega^1_{S/k}$ is isomorphic to $S$.  Also because $S$ is semilocal, thus the image of $u$, as an $S$-submodule, can be generated by one element in the form $1\otimes dy$, where $y=a_1x_1+\cdots+a_nx_n$ and $a_1,\cdots, a_n$ are general elements of the field $k$. Thus we take this $y$ to produce the Noether normalization $\nA=\Spec k[y]$ with the morphism $\mu:X\rightarrow \nA$. Write $\mf{q}:=\mu(\mf{p})$ and $A:=k[y]_{\mf{q}}$. Then in the sequence
$$S\otimes_A\Omega^1_{A/k}\stackrel{u_A}{\longrightarrow} \Omega^1_{S/k}\rightarrow \Omega^1_{S/A}\rightarrow 0,$$
the image of $u_A$ is $S\cdot (1\otimes dy)$. This shows that $\Omega^1_{S/R}=\Omega^1_{S/A}$ and then the statement (1) follows. Again use this element $y$, we can split
$$S\otimes_R\Omega^1_{R/k}=S\cdot (1\otimes dy)\oplus T,$$
where $T$ is the torsion part. We can check that $T=S\otimes_R\Omega^1_{R/A}$. Also from the splitting above we see that $\Fitt^1(S\otimes_R\Omega^1_{R/k})=\Fitt^0(T)$ and then the statement (2) follows. Thus Claim is proved.\\

Having Claim \ref{eq:07} in hand, we see that on $X'_{\mf{p}}$, the condition (\ref{eq:11}) becomes
\begin{equation*}\label{eq:08}
\mf{C}_{X,\mf{p}}\cdot (\Jac(X'/X)\cdot S)=(\Jac_X)_{\mf{p}}\cdot S,
\end{equation*}
which is equivalent to the equality
\begin{equation}\label{eq:03}
\mf{C}_{X,\mf{p}}\cdot(\Jac(X'/\nA)\cdot S)=\Jac(X/\nA)_{\mf{p}}\cdot S.
\end{equation}
We fix such Noether normalization $\mu: X\longrightarrow \nA$ in the sequel.

Let $\mf{C}(X/\nA)$ and $\mf{C}(X'/\nA)$ be the Dedekind complementary modules of $X/\nA$ and $X'/\nA$ (for the details, see \cite[Definition 8.4]{Kunz:ResDuality08}).
We have inclusions $R\subseteq S\subseteq \mf{C}(X'/\nA)_{\mf{p}}\subseteq \mf{C}(X/\nA)_{\mf{p}}$. On the other hand, we have inclusions
\begin{equation}\label{eq:09}
\Jac(X/\nA)_{\mf{p}}\subseteq R:\mf{C}(X/\nA)_{\mf{p}}\subseteq R:\mf(X'/\nA)_{\mf{p}}.
\end{equation}
Since $S$ is principal, $\Jac(X'/\nA)\cdot S=S\cdot y$ for a nonzero divisor $y\in S$.  Because $X'$ is a nonsingular curve, we have $\mf{C}(X'/\nA)=\sO_{X'}:\Jac(X'/\nA)$ and therefore we have $\mf{C}(X'/\nA)_{\mf{p}}=S:(\Jac(X'/\nA)\cdot S)=S\cdot y^{-1}$. Thus $R:\mf{C}(X'/\nA)_\mf{p}=R:S\cdot y^{-1}=(R:S)y=\mf{C}_{X,\mf{p}}\cdot y$. But $\mf{C}_{X,\mf{p}}$ is also an $S$-ideal so $\mf{C}_{X,\mf{p}}\cdot y=\mf{C}_{X,\mf{p}}\cdot (S\cdot y)=\mf{C}_{X,\mf{p}}\cdot (\Jac(X'/\nA)\cdot S)$. Thus we deduce that
$$R:\mf{C}(X'/\nA)_\mf{p}=\mf{C}_{X,\mf{p}}\cdot (\Jac(X'/\nA)\cdot S),$$
and therefore (\ref{eq:03}) is equivalent to
\begin{equation}\label{eq:10}
R:\mf{C}(X'/\nA)_\mf{p}=\Jac(X/\nA)_{\mf{p}}\cdot S.
\end{equation}

Again, since $S$ is principal, we have $\Jac(X/\nA)_{\mf{p}}\cdot S=S\cdot \delta$, where $\delta\in \Jac(X/\nA)_{\mf{p}}$. We show that this $\delta$ generates $\Jac(X/\nA)_{\mf{p}}$ as an $R$-ideal. To this end, consider inclusions of (\ref{eq:09}),
$$R\cdot \delta\subseteq \Jac(X/\nA)_{\mf{p}}\subseteq R:\mf{C}(X/\nA)_{\mf{p}}\subseteq R:\mf{C}(X'/\nA)_{\mf{p}}.$$
By equality (\ref{eq:10}), we have $R:\mf{C}(X'/\nA)_\mf{p}=\Jac(X/\nA)_{\mf{p}}\cdot S=S\cdot \delta$.
Thus we  have
$$R\cdot \delta\subseteq \Jac(X/\nA)_{\mf{p}}\subseteq R:\mf{C}(X/\nA)_{\mf{p}}\subseteq R:\mf{C}(X'/\nA)_{\mf{p}}=S\cdot\delta.$$
Granting the inequality in Claim \ref{eq:05} for the time being,  we immediately have $R\cdot \delta=\Jac(X/\nA)_{\mf{p}}$. This means that $\Jac(X/\nA)$ is principal at $\mf{p}$.

Now by using \cite[Lemma 1]{Lipman:JacobianIdeal}, we see that the projective dimension $\pd_R (\Omega^1_{X/\nA})_{\mf{p}}\leq 1$ since $\Jac(X/\nA)$ is principal at $\mf{p}$. Also notice that we have a short exact sequence
$$0\rightarrow \mu^*\Omega^1_{\nA/k}\rightarrow \Omega^1_{X/k}\rightarrow \Omega^1_{X/\nA}\rightarrow 0,$$
where $\Omega^1_{\nA/k}$ is locally free. Then we obtain that the projective dimension $\pd_R (\Omega^1_{X/k})_{\mf{p}}\leq 1$ and therefore $\pd_{\sO_X}\Omega^1_{X/k}\leq 1$, which implies that $X$ is locally a complete intersection by the well-known result of \cite{Vasconcelos:NormDiffer} or \cite{Ferrand:CompInter}.

Finally, we need to prove the following claim to finish our proof.
\begin{claim}\label{eq:05} Use $\lambda$ to denote the length of modules. Then one has
$$\lambda(\frac{R:\mf{C}(X'/\nA)_{\mf{p}}}{R:\mf{C}(X/\nA)_{\mf{p}}})\ge\lambda(\frac{\mf{C}(X/\nA)_{\mf{p}}}{\mf{C}(X'/\nA)_{\mf{p}}})=\lambda(\frac{S}{R})$$
\end{claim}
{\em Proof.} First note that the last equality can be deduced from duality and Lemma \ref{p:09} by applying $\Hom_R(\ \_\ ,\mf{C}(X/\nA)_{\mf{p}})$ to the exact sequence $0\rightarrow R\rightarrow S\rightarrow S/R\rightarrow 0$. So we just need to show the first inequality.

To this end, we note that there exits a canonical module $K$ of $R$  such that $R\subseteq K\subseteq S$. Indeed, we can first choose a canonical ideal $\omega$ of $R$, then there is an element $t\in \omega$ such that $St=\omega\cdot S$ since $S$ is principal. Then we simply take $K:=t^{-1}\omega$. Also note that such $K$ satisfies the condition $KS=S$.

Now by the choice of $K$, we have $\mf{C}(X/\nA)_{\mf{p}}\subseteq K\mf{C}(X/\nA)_{\mf{p}}$ and $\mf{C}(X'/\nA)_{\mf{p}}=K\mf{C}(X'/\nA)_{\mf{p}}$, which implies that
$$\lambda(\frac{\mf{C}(X/\nA)_{\mf{p}}}{\mf{C}(X'/\nA)_{\mf{p}}})\leq \lambda(\frac{K\mf{C}(X/\nA)_{\mf{p}}}{K\mf{C}(X'/\nA)_{\mf{p}}}).$$
On the other hand, we have $K:K\mf{C}(X'/\nA)_{\mf{p}}=(K:K):\mf{C}(X'/\nA)_{\mf{p}}=R:\mf{C}(X'/\nA)_{\mf{p}}$ since $K:K=R$. Similarly, $K:K\mf{C}(X/\nA)_{\mf{p}}=(K:K):\mf{C}(X/\nA)_{\mf{p}}=R:\mf{C}(X/\nA)_{\mf{p}}$. Thus by Lemma \ref{p:09} we have
$$\lambda(\frac{K\mf{C}(X/\nA)_{\mf{p}}}{K\mf{C}(X'/\nA)_{\mf{p}}})=\lambda(\frac{\Hom_R(K\mf{C}(X'/\nA)_{\mf{p}},K)}{\Hom_R(K\mf{C}(X/\nA)_{\mf{p}},K)})=\lambda(\frac{R:\mf{C}(X'/\nA)_{\mf{p}}}{R:\mf{C}(X/\nA)_{\mf{p}}}),$$
which proves the claim.\\

\end{proof}

\begin{remark}\label{rmk:03} Exactly follow the same argument, Theorem \ref{p:05} can be proved for codimension one points on any variety, i.e., let $X$ be a variety of any dimension and let $\mf{p}\in X$ be a codimension one point, then $X$ is a local complete intersection at $\mf{p}$ if and only if $\widehat{\sI}(\sO_X)_{\mf{p}}=\mf{C}_{X,\mf{p}}$.
\end{remark}

\begin{example}\label{eg:01} (1) Consider $X$ to be the cone over the Segre embedding $\nP^1\times\nP^1\subset \nP^3$ over $\nC$. Then $X$ is MJ-canonical so that $\widehat{\sI}(\sO_X)=\mf{C}_X=\sO_X$. But $X$ is not a local complete intersection. For more details, we refer to \cite[Example 3.13]{Ein:SingMJDiscrapency}.

(2) Consider $X$ to be the cone over a nonsingular hypersurface of degree $d$ in $\nP^{d-1}$ over $\nC$. Then $X$ is a normal locally complete intersection. $X$ is MJ-log canonical but not MJ-canonical. Thus $\widehat{\sI}(\sO_X)$ is not trivial while $\mf{C}_X$ is trivial. See \cite[Example 3.12]{Ein:SingMJDiscrapency} for more details.
\end{example}

At the end of this section, we discuss some evidences for a higher dimensional variety whose MJ-multiplier ideal is the same as its conductor ideal.

\begin{lemma}\label{p:08} Let $X$ be a variety over $k$ and $f:X'\longrightarrow X$ be its normalization. Let $\mf{C}_X$ be the conductor ideal of $X$. Then one has $\mf{C}_X\cdot \omega_X\subseteq f_*\omega_{X'}$. If furthermore $X$ is Gorenstein then one has $\mf{C}_X=f_*\omega_{X'}\otimes\omega^{-1}_X$.
\end{lemma}
\begin{proof} Recall that $\mf{C}_X=\sHom(f_*\sO_{X'},\sO_X)$. Then we have the following diagram
$$\xymatrix{
   &  \sHom(f_*\sO_{X'},\omega_X)=f_*\omega_{X'}\ar[d]\\
\mf{C}_X\otimes\omega_X\ar[r]\ar[ru]    &  \omega_X
}$$
This implies that $\mf{C}_X\cdot \omega_X\subseteq f_*\omega_{X'}$. If $X$ is Gorenstein, then it is by definition that $\mf{C}_X=f_*\omega_{X'}\otimes\omega^{-1}_X$.
\end{proof}

\begin{proposition}\label{p:06} Let $X$ be a Gorenstein variety over $\nC$, $\widehat{\sI}(\sO_X)$ be the MJ-multiplier ideals, and $\mf{C}_X$ be the conductor ideal. Let $X'$ be the normalization of $X$. Assume that $\widehat{\sI}(\sO_X)=\mf{C}_X$. Then one has $\omega_{X'}=\omega^{GR}_{X'}$, where $\omega^{GR}_{X'}$ is the Grauert-Riemenschneider canonical sheaf of $X'$.
\end{proposition}
\begin{proof} Let $f: X'\longrightarrow X$ be the normalization morphism from $X'$ to $X$. By Lemma \ref{p:08}, one has $$\mf{C}_X= f_*\omega_{X'}\otimes \omega^{-1}_X.$$
On the other hand, by Theorem \ref{p:07}, we have $\widehat{\sI}(\sO_X)\subseteq \omega^{GR}_X\otimes \omega^{-1}_X$. Thus the assumption $\widehat{\sI}(\sO_X)=\mf{C}_X$ implies that $\omega^{GR}_X=f_*\omega_{X'}$. Note that $\omega^{GR}_X= f_*\omega^{GR}_{X'}$. So we have $f_*\omega^{GR}_{X'}=f_*\omega_{X'}$. Since $f$ is affine and finite, we deduce that $\omega^{GR}_{X'}=\omega_{X'}$.
\end{proof}

\begin{corollary} Assume that $X$ is a Gorenstein surface over $\nC$ such that $\widehat{\sI}(\sO_X)=\mf{C}_X$. Then the normalization of $X$ has rational singularities.
\end{corollary}
\begin{proof} Let $X'$ be the normalization of $X$. Then by the proposition above, $\omega_{X'}=\omega^{GR}_{X'}$. But $X'$ is Cohen-Macaulay, thus $X'$ has rational singularities.
\end{proof}

\bibliographystyle{alpha}

\end{document}